\documentclass[a4j,12pt,final]{article}
\usepackage{amsmath}
\usepackage{amssymb}
\usepackage{amsthm}
\usepackage{amscd}
\newtheorem{theorem}{Theorem}[section]

\theoremstyle{definition}
\newtheorem{definition}{Definition}[section]
\theoremstyle{remark}
\newtheorem{remark}{Remark}[section]
\theoremstyle{proposition}
\newtheorem{proposition}{Proposition}[section]
\numberwithin{equation}{section}
\newtheorem{example}{Example}[section]
\theoremstyle{corollary}

\theoremstyle{errata}

\begin{document}

\title{A central $U(1)$-extension of a double Lie groupoid}
\author{ Naoya Suzuki}
\date{}
\maketitle
\begin{abstract}
In this paper, we introduce a notion of a central $U(1)$-extension of a double Lie groupoid and show that it defines a cocycle in the certain triple complex.
\end{abstract}

\section{Introduction}
In \cite{Beh}\cite{Beh2}\cite{Tu} and related papers, K. Behrend, J.-L. Tu, P. Xu and C. Laurent-Gengoux have developed a theory of central $U(1)$-extensions of 
Lie groupoids. On the other hand, there is a theory of double Lie groupoids due to R.Brown, K.C.H. Mackenzie, Mehta Rajan Amit and Tang, X \cite{Bro}\cite{Mac}\cite{Met}. In this paper, we combine these two theories and introduce a notion of a central $U(1)$-extension of a double Lie groupoid.
Given a double Lie groupoid, we can construct a bisimplicial manifold and a triple complex on it. If a double Lie groupoid have a central $U(1)$-extension, we show that we can construct a $3$-cocycle in the
triple complex.
\section{
Lie groupoids and double complexes}

At first we recall the definition of Lie groupoids following \cite{Mik}.
\begin{definition}
A Lie groupoid $\Gamma _1$ over a manifold  $\Gamma _0$ is a pair $(\Gamma _1,\Gamma _0)$
equipped with following differentiable maps:\\
(i) surjective submersions $s,t:\Gamma _1 \rightarrow \Gamma _0$ called the source and target maps respectively;\\
(ii) $m:\Gamma _2 \rightarrow \Gamma _1$ called multiplication, where $\Gamma _2 := \{ (x_1,x_2) \in \Gamma _1 \times \Gamma _1 |\enspace t(x_{1})=s(x_{2}) \}$;\\
(iii) an injection $e:\Gamma _0 \rightarrow \Gamma _1$ called identities;\\
(iv) $\iota : \Gamma _1 \rightarrow \Gamma _1$ called inversion.\\
These maps must satisfy:\\
(1) (associative law) $m(m(x_1,x_2),x_3)=m(x_1,m(x_2,x_3))$ if one is defined, so is the other;\\
(2) (identities) for each $x \in \Gamma _1, (e(s(x)),x) \in \Gamma _2, (x,e(t(x)))  \in \Gamma _2$ and $m(e(s(x)),x)=m(x,e(t(x)))=x$;\\
(3) (inverses) for each $x \in \Gamma _1, (x,\iota(x)) \in \Gamma _2, (\iota(x),x)  \in \Gamma _2, m(x,\iota(x))=e(s(x))$, and $m(\iota(x),x)=e(t(x))$.
\end{definition}
In this paper we denote a Lie groupoid by $\Gamma _1 \rightrightarrows \Gamma _0$.

\begin{example}
{Suppose that} $G$ is a Lie group. Then we have a Lie groupoid $G \rightrightarrows *$.
\end{example}

\begin{example}
{Suppose that} $M$ is a manifold. Then we have a Lie groupoid $\Gamma _1=M \times M$, $\Gamma _0 =M$. The source map $s: M \times M \rightarrow M$ is a projection map into the
first factor and  the target map $t$ is a projection into the second factor.
This groupoid $M \times M \rightrightarrows M$
is called a pair groupoid.
\end{example}

\begin{example}
{Suppose that} $G$ is a Lie group acting on a manifold $M$ by left. Then we have a Lie groupoid $\Gamma _1=G \times M$, $\Gamma _0 =M$.
The source map $s$ is defined as $s(g,u)=u$ and  the target map $t$ is defined as $t(g,u)=gu$. This groupoid $M  \rtimes G \rightrightarrows M$
is often called an action groupoid.
\end{example}
\begin{example}
{Suppose that} $M$ is a manifold and $\{ U_{\alpha } \}$ is a covering of $M$. Then we have a Lie groupoid $\Gamma _1=\coprod (U_{\alpha} \cap U_{\beta}) $,  $\Gamma _0 =\coprod U_{\alpha}$.
The source map $s$ is an inclusion map into $U_{\alpha}$ and  the target map $t$ is  an inclusion map into $U_{\beta}$.
\end{example}

Let $\Gamma _1 \rightrightarrows \Gamma _0$ be a Lie groupoid and denote by $s,t,m$ the source and target maps, and the multiplication of it respectively. Then we 
can define a simplicial manifold $N\Gamma$ as follows:
$$N\Gamma(p)  := \{ (x_1, \cdots , x_p) \in \overbrace{\Gamma_1 \times \cdots \times \Gamma_1 }^{p-times}  \enspace | \enspace t(x_{j})=s(x_{j+1}) \enspace j=1 , \cdots , p-1 \},$$  
face operators \enspace ${\varepsilon}_{i} : N\Gamma(p) \rightarrow N\Gamma(p-1)  $
$$
{\varepsilon}_{i}(x_1 , \cdots , x_p )=\begin{cases}
(x_2 , \cdots , x_p )  &  i=0 \\
(x_1 , \cdots ,m(x_i ,x_{i+1}) , \cdots , x_p )  &  i=1 , \cdots , p-1 \\
(x_1 , \cdots , x_{p-1} )  &  i=p.
\end{cases}
$$

Then recall how to construct a double complex associated to a simplicial manifold.

\begin{definition}
For any simplicial manifold $ \{ X_* \}$ with face operators $\{ {\varepsilon}_* \}$, we define a double complex as follows:
$${\Omega}^{p,q} (X) := {\Omega}^{q} (X_p).$$
Derivatives are:
$$ d' := \sum _{i=0} ^{p+1} (-1)^{i} {\varepsilon}_{i} ^{*}  , \qquad  d'' := (-1)^{p} \times {\rm the \enspace exterior \enspace differential \enspace on \enspace }{ \Omega ^*(X_p) }. $$
\end{definition}

To any simplicial manifold $\{ X_* \}$, we can associate a topological space $\parallel X_* \parallel $ 
called the fat realization.

\begin{theorem}[\cite{Bot2} \cite{Dup2} \cite{Mos}]
{There exists a ring isomorphism}
$$ H({\Omega}^{*} (X_*))  \cong  H^{*} (\parallel X_* \parallel).$$
Here ${\Omega}^{*} (X_*)$ means the total complex. 
\end{theorem}

\begin{example}
{In the case of an action groupoid} $M  \rtimes G \rightrightarrows M$ for a compact Lie group $G$,
$H({\Omega}^{*} (N\Gamma)) $ is isomorphic to the Borel model of the equivariant cohomology $H^* _G(M) := H^*(EG \times _{G} M)$.
\end{example}
\begin{example}
{In the case of the groupoid }$\coprod (U_{\alpha} \cap U_{\beta}) \rightrightarrows \coprod U_{\alpha}$ for a good covering $\{ U_{\alpha } \}$ in example 2.4, $H({\Omega}^{*} (N\Gamma)) $ is isomorphic to $H^*(M)$.\\
\end{example}

\section{A central $U(1)$-extension of a Lie groupoid}

Now we recall the notion of a central $U(1)$-extension of a Lie groupoid in \cite{Beh2}.

\begin{definition}
A central $U(1)$-extension of a Lie groupoid $\Gamma _1 \rightrightarrows \Gamma _0$ consists of a morphism of Lie groupoids
$$
\begin{array}{ccc}
{\widehat{\Gamma_1}}&{\rightrightarrows}&{{\Gamma_0}}\\
{\downarrow{\pi}}&&{\downarrow{\rm id}}\\
{\Gamma_1}&{\rightrightarrows}&{\Gamma_0}
\end{array}
$$
and a right $U(1)$-action on ${\widehat{\Gamma}_1}$, making $\pi:{\widehat{\Gamma}_1} \rightarrow {\Gamma_1}$ a principal $U(1)$-bundle. 
For any $z_1,z_2 \in U(1)$ and $(\hat{x_1},
\hat{x_2}) \in N \widehat{\Gamma}(2) := \{ (\hat{y_1},\hat{y_2}) \in \widehat{\Gamma}_1 \times \widehat{\Gamma}_1 |~ t(\hat{y_1}) =
s(\hat{y_2}) \}$, the equation $\hat{m}( \hat{x_1}z_1,  \hat{x_2}z_2) =  \hat{m}(\hat{x_1} , \hat{x_2})z_1z_2$ holds.
\end{definition}

Let $\widehat{\Gamma}_1 \rightarrow \Gamma _1 \rightrightarrows \Gamma _0$ be a central $U(1)$-extension of a Lie groupoid 
$\Gamma _1 \rightrightarrows \Gamma _0$. Using the face operators $\{ {\varepsilon}_{i} \}: N\Gamma(2) \rightarrow \Gamma_1 $, 
we can construct the $U(1)$-bundle over $N\Gamma(2)$ as $\delta \widehat{\Gamma}_1 :=  {\varepsilon _0}^* \widehat{\Gamma}_1 \otimes ({{\varepsilon} _1}^* \widehat{\Gamma}_1)^{{\otimes}-1} \otimes
{{\varepsilon} _2}^* \widehat{\Gamma}_1 $.
Here we define the tensor product $S \otimes T$ of $U(1)$-bundles $S$ and $T$ over $M$ as:
$$ S \otimes T := \bigcup _{x \in M} (S_x \times T_x) / (s,t) \sim (su,tu^{-1}),~~ (u \in U(1)).$$

Note that there is a natural section $\hat{s}_{nt}$ of 
$\delta \widehat{\Gamma}_1 $  defined as: 
$$\hat{s}_{nt}(x_1,x_2):=[((x_1,x_2),\hat{x}_2 ) , \\ ((x_1,x_2),\hat{m}(\hat{x} _1 ,\hat{x} _2)  )^{\otimes -1},((x_1,x_2),\hat{x}_1 ) ].$$ 
Furthermore, because of the associative law of $\Gamma _1 \rightrightarrows \Gamma _0$,
$\delta(\delta \widehat{\Gamma}_1))$ is canonically isomorphic to the product bundle and $\delta \hat{s} _{nt}=1$ holds.\\

A bundle gerbe, which is introduced by M.K.Murray\cite{Mu}, is a very important example of central $U(1)$-extensions of Lie groupoids.
\begin{definition}[Murray-Stevenson, \cite{Mu}\cite{MS}]
Given a surjective submersion $\phi : Y \rightarrow M $, we obtain the groupoid $Y^{[2]} \rightrightarrows Y$ where
$Y^{[2]}$ is the fiber product defined as $Y^{[2]}:=\{(y_1,y_2)| \phi(y_1)= \phi(y_2) \}$. The source and target maps are defined as $s(y_1,y_2)=y_2,
t(y_1,y_2)=y_1$ respectively.

A bundle gerbe over $M$ is a
 pair of $\phi : Y \rightarrow M $, a principal $U(1)$-bundle $\widehat{Y^{[2]}}$ over $Y^{[2]}$ and 
  a section $\hat{s}$ of $\delta \widehat{Y^{[2]}}$ which satisfies $\delta \hat{s} =1$.\\
\end{definition}

\begin{remark}
Without the assumption of the existence of $\hat{s}$, $\delta \widehat{Y^{[2]}}$ is not necessarily trivial.
By using $\hat{s}$, we can construct a multiplication $\hat{m}:\widehat{Y^{[2]}} \times \widehat{Y^{[2]}} \rightarrow \widehat{Y^{[2]}}$
such that $\hat{s}$ is a natural section of $\delta \widehat{Y^{[2]}}$. Hence we can recognize  bundle gerbe as a kind of a central $U(1)$-extension of a Lie groupoid.\\
\end{remark}

\section{A cocycle in the double complex}

For any connection $\theta$ on a $U(1)$-bundle $\widehat{\Gamma}_1 \rightarrow {\Gamma}_1$, there is the induced connection $\delta \theta $ on $\delta \widehat{\Gamma}_1 $.

\begin{proposition}
{Let} $c_1 (\theta )$ denote the 2-form on $\Gamma_1$ which hits $\left( \frac{-1}{2 \pi i} \right) d \theta \in \Omega ^2 (\widehat{\Gamma}_1)$ by $\pi ^{*}$, and $\hat{s}$ a global section of
$\delta \widehat{\Gamma}_1 $ such that $\delta \hat{s}:= {\varepsilon}_{0} ^* \hat{s} \otimes ({\varepsilon}_{1} ^* \hat{s})^{\otimes -1} \otimes {\varepsilon}_{2} ^* \hat{s} \otimes ({\varepsilon}_{3} ^* \hat{s})^{\otimes -1}=1$. Then the following equations hold.
$$ ({\varepsilon}_{0} ^* - {\varepsilon}_{1} ^* + {\varepsilon}_{2} ^* ) c_1 (\theta ) = \left( \frac{-1}{2 \pi i} \right)d(\hat{s}^{*} (\delta \theta )) \enspace \in \Omega ^2 (N\Gamma(2)) ,$$ 
$$({\varepsilon}_{0} ^* - {\varepsilon}_{1} ^* + {\varepsilon}_{2} ^* - {\varepsilon}_{3} ^*)(\hat{s}^{*} (\delta \theta ))=0.$$
\end{proposition}

\begin{proof}
See \cite{Mur}\cite{Mur2} or \cite{Suz}.
\end{proof}

\par

The proposition above give the cocycle below.

$$
\begin{CD}
0 \\
@AA{-d}A \\
c_1 ( \theta ) \in {\Omega}^{2} (\Gamma_1 )@>{{\varepsilon}_{0} ^* - {\varepsilon}_{1} ^* + {\varepsilon}_{2} ^* }>>{\Omega}^{2} (N\Gamma(2))\\
@.@AA{d}A\\
@.-\left( \frac{-1}{2 \pi i} \right) \hat{s}^{*} (\delta \theta ) \in {\Omega}^{1} (N\Gamma(2))@>{{\varepsilon}_{0} ^* - {\varepsilon}_{1} ^* + {\varepsilon}_{2} ^* - {\varepsilon}_{3} ^*}>> 0.
\end{CD}
$$

Let $\widehat{\Gamma}_1 \rightarrow \Gamma _1 \rightrightarrows \Gamma _0$ be a central $U(1)$-extension of a groupoid and
$\theta$ be a connection form of the $U(1)$-bundle $\widehat{\Gamma}_1 \rightarrow \Gamma _1$.  In \cite{Beh}\cite{Beh2} and 
related papers, \\
K. Behrend and P. Xu add $B \in \Omega^2(\Gamma_0)$ to $\theta$ and call $\theta +B$ a pseudo-connection of $\widehat{\Gamma}_1 \rightarrow \Gamma _1 \rightrightarrows \Gamma _0$.
$$
\begin{CD}
0 \\
@AA{d}A \\
{ {\Omega}^{3} (\Gamma_0 )}@>{d'}>>*\\
@AA{d}A@AA{-d}A\\
{B \in {\Omega}^{2} ({\Gamma}_0)}@>{d'}>>{ {\Omega}^{2} (\widehat{\Gamma}_1)}@>{d'}>> * \\
@.@AA{-d}A@AA{d}A\\
@.{{\theta} \in {\Omega}^{1} (\widehat{\Gamma}_1)}@>{d'}>>{ {\Omega}^{1} (N\widehat{\Gamma}(2))}@>{d'}>> 0.
\end{CD}
$$

\begin{theorem}[\cite{Beh}\cite{Beh2}]
For any pseudo-connection $B+\theta$, there exists a $3$-cocycle $\mu + \omega + \eta,~\mu \in {\Omega}^{3} (\Gamma_0 ), \omega \in {\Omega}^{2} (\Gamma_1),\eta \in {\Omega}^{1} (N\Gamma_1(2)) $ which 
satisfies $\pi^*(\mu + \omega + \eta)=(d'+d'')(B+\theta)$. This $3$-cocycle is called a pseudo-curvature of $B+\theta$.
\end{theorem}

$$
\begin{CD}
0 \\
@AA{d}A \\
{ \mu \in {\Omega}^{3} (\Gamma_0 )}@>{d'}>>{\Omega}^{3} (\Gamma_1)\\
@.@AA{-d}A\\
@.{ \omega \in {\Omega}^{2} (\Gamma_1)}@>{d'}>> {\Omega}^{2} (N\Gamma(2)) \\
@.@.@AA{d}A\\
@.@.{ \eta \in {\Omega}^{1} (N\Gamma(2))}@>{d'}>> 0.
\end{CD}
$$

\begin{theorem}[\cite{Beh}\cite{Beh2}]
Assume that $\Gamma_1 \rightrightarrows \Gamma _0$ is a proper Lie groupoid. Given any $3$-cocycle $\mu + \omega + \eta \in Z^3_{dR}(\Gamma_*)$
such that

(1)~$[\mu + \omega + \eta]$ is integral, and

(2)~$\mu$ is exact.

Then there exists a Lie groupoid $S^1$-central extension $\widehat{\Gamma}_1 \rightarrow \Gamma _1 \rightrightarrows \Gamma _0$
and a pseudo-connection $\theta +B \in \Omega^1(\widehat{\Gamma}_1) \oplus \Omega^2(\Gamma_0)$ such that its pseudo-curvature is equal to $\mu + \omega + \eta$.
\end{theorem}

\begin{proposition}
We take $\hat{s}=\hat{s}_{nt}$ and $B=0$, then the following equations hold.
 $$\omega=c_1 ( \theta ),$$
$$\eta=-\left( \frac{-1}{2 \pi i} \right) \hat{s}_{nt}^{*} (\delta \theta ).$$
\end{proposition}

\begin{proof}
See \cite{Suz}.
\end{proof}

\section{Double Lie groupoids and triple complexes}

We recall the definition of double Lie groupoids following \cite{Met}.

\begin{definition}

$$
\begin{array}{ccc}
{{\Gamma^V_1}}&{\leftleftarrows}&{\bar{{\Gamma}}}\\
{\downdownarrows}&&{\downdownarrows}\\
{\Gamma_0}&{\leftleftarrows}&{\Gamma_1^H}
\end{array}
$$
The square above is a double Lie groupoid if the following conditions hold:

(1)~The horizontal and vertical source and target maps commute:
$$s \circ s_H=s \circ s_V,~~~~~~~~~t \circ t_H= t \circ t_V$$
$$t \circ s_H= s \circ t_V,~~~~~~~~~s \circ t_H= t \circ s_V.$$
Here we use $s$ and $t$ to denote the source and target maps from either ${\Gamma^V_1}$ or ${\Gamma^H_1}$ to $\Gamma_0$.\\

(2)~The following equations hold:
$$s_V(m_H(x_1,x_2))=m(s_V(x_1),s_V(x_2)),~~~~~t_V(m_H(x_1,x_2))=m(t_V(x_1),t_V(x_2)),$$
$$s_H(m_V(x_1,x_2))=m(s_H(x_1),s_H(x_3)),~~~~~t_H(m_V(x_1,x_3))=m(s_H(x_1),s_H(x_2)),$$
for all $x_i \in \bar{\Gamma}$ such that $s_Hx_1=t_Hx_2$ and $s_Vx_1=t_Vx_3$.\\

(3)~The interchange law
$$m_V(m_H(x_{11},x_{12}),m_H(x_{21},x_{22}))=m_H(m_V(x_{11},x_{21}),m_V(x_{12},x_{22}))$$
holds for $x_{11}, x_{12}, x_{21}, x_{22} \in \bar{\Gamma}$ such that $s_H(x_{i1})=t_H(x_{i2})$ and $s_V(x_{1i})=t_V(x_{2i})$
for $i=1,2$.\\

(4)~The double-source map $(s_V,s_H):\bar{\Gamma} \rightarrow {\Gamma^V_1} _s\times_s {\Gamma^H_1}$ is a submersion,
where ${\Gamma^V_1} _s\times_s {\Gamma^H_1}:=\{(x_1,x_2)\in {\Gamma^V_1} \times {\Gamma^H_1}~|~s(x_1)=s(x_2) \}$.

\end{definition}

\begin{example}
For any smooth manifold $M$, we have a following double Lie groupoid: 
$$
\begin{array}{ccc}
{M}&{\leftleftarrows}&{M}\\
{\downdownarrows}&&{\downdownarrows}\\
{M}&{\leftleftarrows}&{M}
\end{array}
$$
\end{example}

\begin{example}
Let $G$ be a Lie group and $H$ be a subgroup of $G$.
We define the product of $G \rtimes H$ as: $(g_1,h_1) \rtimes (g_2,h_2):= (g_1h_1g_2h^{-1}_1, h_1h_2)$.
Then we have a following double Lie groupoid: 
$$
\begin{array}{ccc}
{H}&{\leftleftarrows}&{G \rtimes H}\\
{\downdownarrows}&&{\downdownarrows}\\
{*}&{\leftleftarrows}&{G}
\end{array}
$$
\end{example}

\begin{example}
Let $\Gamma_1 \rightrightarrows \Gamma_0$ be a Lie groupoid.
Then the following square is a double Lie groupoid.
$$
\begin{array}{ccc}
{{\Gamma_1}}&{\leftleftarrows}&{{\Gamma}_1 \times {\Gamma}_1}\\
{\downdownarrows}&&{\downdownarrows}\\
{\Gamma_0}&{\leftleftarrows}&{\Gamma_0 \times \Gamma_0}
\end{array}
$$
Here the horizontal edges are pair groupoids. This double groupoid is called a pair double
groupoid.
\end{example}

\begin{example}
Let $\Gamma_1 \rightrightarrows \Gamma_0$ be a $G$-groupoid, i.e. both $\Gamma_1$ and $\Gamma_0$ are 
$G$-manifolds and all structure maps are $G$-equivariant \cite{Gin}. Then the following square is a double Lie groupoid.
$$
\begin{array}{ccc}
{{\Gamma_1}}&{\leftleftarrows}&{{\Gamma}_1 \times G}\\
{\downdownarrows}&&{\downdownarrows}\\
{\Gamma_0}&{\leftleftarrows}&{\Gamma_0 \times G}
\end{array}
$$
Here the horizontal edges are action groupoids.
\end{example}

You can find more examples in \cite{Met}.\\

A bisimplicial manifold $\{ X_{*,*} \}$ is a sequence of manifolds with horizontal and vertical face and degeneracy operators which commute with each other.

Given a double Lie groupoid, we can define a bisimplicial manifold as follows\cite{Met}:

For $p,q \ge 1$,
$$N\bar{{\Gamma}}(p,q)  := \{ (x_{ij})_{1 \le i \le p, 1 \le j \le q} |$$
$$ x_{ij} \in \bar{\Gamma},  t_H(x_{ij})=s_H(x_{i(j+1)}),  \enspace t_V(x_{ij})=s_V(x_{(i+1)j})\}$$
For $q=0$, $N\bar{{\Gamma}}(p,0):=N\Gamma_1^H(p)$ and for $p=0$, $N\bar{{\Gamma}}(0,q):=N\Gamma_1^V(q)$.

\begin{definition}
For a bisimplicial manifold $\{ X_{*,*} \}$, we can construct a triple complex on it in the following way:

$${\Omega}^{p,q,r} (X_{*,*}) := {\Omega}^{r} (X_{p,q}) $$

Derivatives are:
$$ d' := \sum _{i=0} ^{p+1} (-1)^{i} ({{\varepsilon}^{Ho} _{i}}) ^{*}  , \qquad  d'' := \sum _{i=0} ^{q+1} (-1)^{i} ({{\varepsilon}^{Ve} _{i}}) ^{*} \times (-1)^{p}, $$
$$ d''' :=  (-1)^{p+q} \times {\rm the \enspace exterior \enspace differential \enspace on \enspace }{ \Omega ^*(X_{p,q}) }.$$
Here ${\varepsilon}^{Ho} _{i}$ and ${\varepsilon}^{Ve} _i$ are horizontal and vertical face operators of $\{X_{*,*}\}$.
\end{definition}

\section{
A central $U(1)$-extension of a double Lie groupoid}

\begin{definition}
A central $U(1)$-extension of a double Lie groupoid consists of $U(1)$-extensions of Lie groupoids
$$
\begin{array}{ccc}
{\widehat{\Gamma_1^V}}&{\rightrightarrows}&{{\Gamma_0}}\\
{\downarrow{\pi_V}}&&{\downarrow{\rm id}}\\
{\Gamma_1^V}&{\rightrightarrows}&{\Gamma_0}
\end{array}
$$
and 
$$
\begin{array}{ccc}
{\widehat{\Gamma_1^H}}&{\rightrightarrows}&{{\Gamma_0}}\\
{\downarrow{\pi_H}}&&{\downarrow{\rm id}}\\
{\Gamma_1^H}&{\rightrightarrows}&{\Gamma_0}
\end{array}
$$
and a section $\bar{s}$ of $U(1)$-bundle $\delta \widehat{\bar{\Gamma}}:= {s_V}^*\widehat{\Gamma_1^V }\otimes ({t_V}^*\widehat{\Gamma_1^V })^{\otimes -1} \otimes ({s_H}^*\widehat{\Gamma_1^H })^{\otimes-1} \otimes {t_H}^*\widehat{\Gamma_1^H }
\rightarrow \bar{\Gamma}$
which satisfies the following conditions:

(1)~$\delta \bar{s}=\delta s^H_{nt}$ on $\delta(\delta \widehat{\bar{\Gamma}}) \rightarrow N\bar{\Gamma}(2,1)$.

(2)~$\delta \bar{s}=\delta s^V_{nt}$ on $\delta(\delta \widehat{\bar{\Gamma}}) \rightarrow N\bar{\Gamma}(1,2)$.

\end{definition}

Here $\delta s^H_{nt}$ is a natural section of $\delta \widehat{\Gamma}_1^H \rightarrow N\bar{\Gamma}(2,0)$ and  $\delta s^V_{nt}$ is a natural section of $\delta \widehat{\Gamma}_1^V \rightarrow N\bar{\Gamma}(0,2)$.

\begin{example}
It is well-known that the free loop group of $SU(2)$ have a non-trivial central $U(1)$-extension \cite{Bry}\cite{Pre}.
So the following double Lie groupoid have a central $U(1)$-extension.
$$
\begin{array}{ccc}
{LSU(2)}&{\leftleftarrows}&{LSU(2) \rtimes LSU(2)}\\
{\downdownarrows}&&{\downdownarrows}\\
{*}&{\leftleftarrows}&{LSU(2)}
\end{array}
$$
\end{example}

\begin{example}
Let $\Gamma_1 \rightrightarrows \Gamma_0$ be a $G$-groupoid assume that it have a central $U(1)$-extension. 
Then the double Lie groupoid in the example 5.4 have a central $U(1)$-extension. 
\end{example}

\section{A cocycle in the triple complex}

\begin{definition}
Let $\theta^V \in \Omega^1(\widehat{\Gamma_1^V})$ be a connection of $U(1)$-bundle $\widehat{\Gamma_1^V} \rightarrow {\Gamma_1^V}$ and
$\theta^H \in \Omega^1(\widehat{\Gamma_1^H})$ be a connection of $U(1)$-bundle $\widehat{\Gamma_1^H} \rightarrow {\Gamma_1^H}$.
We call $\theta^V + \theta^H$ a connection of a central $U(1)$-extension of a double Lie groupoid.

\end{definition}

\begin{theorem}
For any central $U(1)$-extension of a double Lie groupoid, we can construct a $3$-cocycle $\omega^V + \omega^H + \bar{\eta}+\eta^V+\eta^H \in \Omega^3(N\bar{\Gamma}),~\omega^V \in {\Omega}^{2} (\Gamma_1^V ), \omega^H \in {\Omega}^{2} (\Gamma_1^H),\bar{\eta} \in {\Omega}^{1} (\bar{\Gamma}), \eta^V \in \Omega^1(N\Gamma^V(2)), \eta^H \in \Omega^1(N\Gamma^H(2)) $.
\end{theorem}

\begin{proof}
We set:
$$\omega^V:=c_1 ( \theta^V ),~~~~\omega^H:=-c_1 ( \theta^H )$$
$$\eta^V:=-\left( \frac{-1}{2 \pi i} \right) {s^V_{nt}}^{*} (\delta \theta^V ),~~~~\eta^H:=\left( \frac{-1}{2 \pi i} \right) {s^H_{nt}}^{*} (\delta \theta^H ).$$
Then $d''\omega^V + d'''\mu^V=0$ and $d'\omega^H + d'''\mu^H=0$ follows from the proposition 4.1.

We define $\bar{\eta}$ as:
$$\bar{\eta}:=-\left( \frac{-1}{2 \pi i} \right) \bar{s}^*\delta(\delta \theta^H \delta \theta^V).$$
Then we can see that the equations $d'\bar{\eta}+d''\eta^H=0$, $d''\bar{\eta}+d'\eta^V=0$, $d'''\bar{\eta}+d'\omega^V+d''\omega^H=0$ hold because the properties of $\bar{s}$, therefore $\omega^V + \omega^H + \bar{\eta}+\eta^V+\eta^H$ is a cocycle.
\end{proof}
\begin{remark}
If the following double Lie groupoid exists,
$$
\begin{array}{ccc}
\widehat{{\Gamma^V_1}}&{\leftleftarrows}&\widehat{\bar{{\Gamma}}}\\
{\downdownarrows}&&{\downdownarrows}\\
\widehat{\Gamma_0}&{\leftleftarrows}&\widehat{\Gamma_1^H}
\end{array}
$$
and $s$ have a kind of naturality, i.e. $\displaystyle -\left( \frac{-1}{2 \pi i} \right)(d'\theta^V+d''\theta^H)= \pi^*\bar{\eta}$ holds,
then the following equation holds true:
$${\pi_V}^*(\omega^V + \mu^V)+{\pi_H}^*(\omega^H + \mu^H)+\bar{\pi}^*\bar{\eta} =-\left( \frac{-1}{2 \pi i} \right)(d'+d''+d''')(\theta^V+ \theta^H).$$
\end{remark}


\begin{thebibliography}{99}
\bibitem{Suz4}N. Suzuki, On some cocycles which represent the Dixmier-Douady class in simplicial de Rham complexes, math.DG/13104559.
\end{thebibliography}
\end{document}